\documentclass[11pt,centertags]{amsart}
\usepackage{amssymb}            
\usepackage{enumerate}          
\usepackage{fancyhdr}           
\usepackage{url}                
\usepackage{color}
\usepackage{graphicx}

\numberwithin{equation}{section}

\newcommand{\comm}[1]{}
\newtheorem{theorem}{Theorem}
\newtheorem{definition}[theorem]{Definition}
\newtheorem{lemma}[theorem]{Lemma}
\newtheorem{question}[theorem]{Question}
\newtheorem{remark}[theorem]{Remark}
\newtheorem{proposition}[theorem]{Proposition}
\newtheorem{corollary}[theorem]{Corollary}

\numberwithin{theorem}{section}

\theoremstyle{remark}

\DeclareMathOperator{\GL}{GL}

\newcommand{\ZZ}{{\mathbb{Z}}}
\newcommand{\NN}{{\mathbb{N}}}

\newcommand{\interior}[1]{%
  {\kern0pt#1}^{\mathrm{o}}%
}

 \newcounter{case}
 
 \renewcommand{\thecase}{\arabic{case}}

\newcommand{\op}[1]{\operatorname{#1}}

\newcommand{\floor}[1]{\left\lfloor #1 \right\rfloor}
\newcommand{\union}{\cup}

\providecommand{\to}{\longrightarrow }

		\renewcommand{\bf}{\bfseries}

		\renewcommand{\it}{\itshape}

		\renewcommand{\bar}{\overline}
		
\newcommand{\cout}[1]{}

\begin{document}

\author[Thang Nguyen]{Thang Nguyen}
\author[Shi Wang]{Shi Wang}

\title[Splitting Theorem]{Cheeger-Gromoll Splitting Theorem for groups}

\address{University of Michigan, Ann Arbor, Michigan, US} 
\email{thang.q.nguyen7@gmail.com} 

\address{Max Planck Institute for Mathematics, Bonn, Germany} 
\email{shiwang.math@gmail.com}

\begin{abstract}
We study a notion of curvature for finitely generated groups which serves as a role of Ricci curvature for Riemannian manifolds. We prove an analog of Cheeger-Gromoll splitting theorem. As a consequence, we give a geometric characterization of virtually abelian groups. We also explore the relation between this notion of curvature and the growth of groups.
\end{abstract}

\maketitle


\thispagestyle{empty} 

\section[]{Introduction}

In differential geometry, the notion of Ricci curvature measures the extend to which the volume of infinitesimal conical geodesic ball deviates from that in the Euclidean spaces. Studying the structures of Riemannian manifolds with Ricci curvature bounded from below has a long history and yields many celebrated results. This includes Bishop-Gromov inequality \cite{bishop,Gromov81book}, Gromov's almost flat manifold theorem \cite{almostflat}, Gromov compactness theorem \cite{Gromov81book}, Cheeger-Colding and Cheeger-Naber structure theorems \cite{CheegerColdingI,CheegerNaber}, as well as Cheeger-Gromoll splitting theorem \cite{CG71} which states as follows,

\begin{theorem}
Let $M$ be a complete Riemannian manifold with non-negative Ricci curvature. If $M$ contains a bi-infinite geodesic, then $M$ splits isometrically as $N\times \mathbb R$.
\end{theorem}

As a consequence, the fundamental group of a closed Riemannian manifold of non-negative Ricci curvature is virtually abelian. Conversely, Wilking \cite{Wil00} showed that any such group can be realized by the fundamental group of a closed manifold of non-negative Ricci curvature, or even non-negative sectional curvature. The main goal of this paper is to establish similar results for finitely generated groups. Nevertheless, before we can make precise of the statement, we need an appropriate notion of Ricci curvature for groups.

The key observation is the following equivalent characterization of Ricci curvature for Riemannian manifolds (see \cite[Corollary 10]{Oll10}). Let $M$ be an $n$-dimensional smooth manifold with a Riemannian distance $d$, and let $x$ and $ y$ be a pair of points in $M$. We denote $\gamma$ the unit speed geodesic connecting $x$ to $y$. Assume $d(x,y)$ is small, the average distance (with respect to the Lebesgue measure) between points in $S_x(\epsilon)$ to points in $S_y(\epsilon)$ under the identification via the {\it Riemannian parallel transport} along $\gamma$  is given by
$$d(x,y)\bigg(1-\frac{\epsilon^2}{2n}\op{Ric}(\dot{\gamma})+O(\epsilon^3+\epsilon^2d(x,y))\bigg),$$
where $S_x(\epsilon)$ and $S_y(\epsilon)$ are the $\epsilon$-spheres at $x$ and $y$. Thus, a manifold has positive Ricci curvature if on average small balls are closer than their centers are. 

\begin{figure}[ht]
		\centering
		\includegraphics[width=0.7\linewidth]{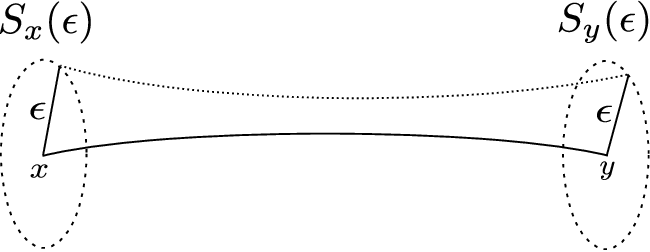}
		\label{fig:Ric-defn-manifold}
\end{figure}

Motivated by this fact, recently, Ollivier \cite{Oll07, Oll09, Oll10, Oll13} introduced the notion of discrete Ricci curvature on graphs, and other more general type of metric spaces. Ollivier's idea is to use the {\it optimal transport} as a replacement for the Riemannian parallel transport. Given any pair of points $x$ and $y$ in the space, up to a normalization, the curvature in direction from $x$ to $y$ is defined to be the difference between the cost to carry the Dirac mass at $x$ to the Dirac mass at $y$ (which is $d(x,y)$) and the \emph{optimal} cost to carry the uniformly distributed probability measure on $S_x(\epsilon)$ to that on $S_y(\epsilon)$. We refer the readers to \cite{Oll09} for a more precise definition in a more general setting.

More recently, based on a similar idea, Bar-Natan--Duchin--Kropholler \cite{NDK} introduced a notion of Ricci curvature for finitely generated groups, by suggesting the {\it left translation} as a parallel transport. Let $\Gamma$ be a group with a finite symmetric generating set $S$, the Cayley graph is a metric space with the left $\Gamma$-invariant word metric. Given any pair of points $x, y\in \Gamma$, we can measure the average distance between the $1$-spheres $S_x(1), S_y(1)$ under the identification via the left translation by $yx^{-1}$, that is,
$$\mathcal{S}_1(x,y)=\frac{1}{|S|}\sum_{s\in S}d(xs,ys).$$
Since the metric is left invariant, it is clear that $\mathcal{S}_1(x,y)=\mathcal{S}_1(1,x^{-1}y)$. Thus, it is sufficient to define the curvature from the identity to any other element in the group. Since heuristically Ricci curvature measures the difference between the distance of the centers $|\gamma|$ and the average distance $\mathcal{S}_1(1,\gamma)$, we can make the following definition.

\begin{definition}
	Let $\Gamma$ be a finitely generated group with symmetric generating set $S$, for every $\gamma\in \Gamma$, we define the \emph{curvature} of $\gamma$ to be,
	$$\kappa_S(\gamma)=\frac{1}{|S|}\sum_{s\in S}(d(1,\gamma)-d(s,\gamma s))=\frac{1}{|S|}\sum_{s\in S}(|\gamma|-|s\gamma s^{-1}|),$$
	where $|S|$ is the carnality of $S$, and $|\gamma|$ is the distance from $\gamma$ to the identity under the word metric.
\end{definition}

\begin{figure}[ht]
		\centering
		\includegraphics[width=0.7\linewidth]{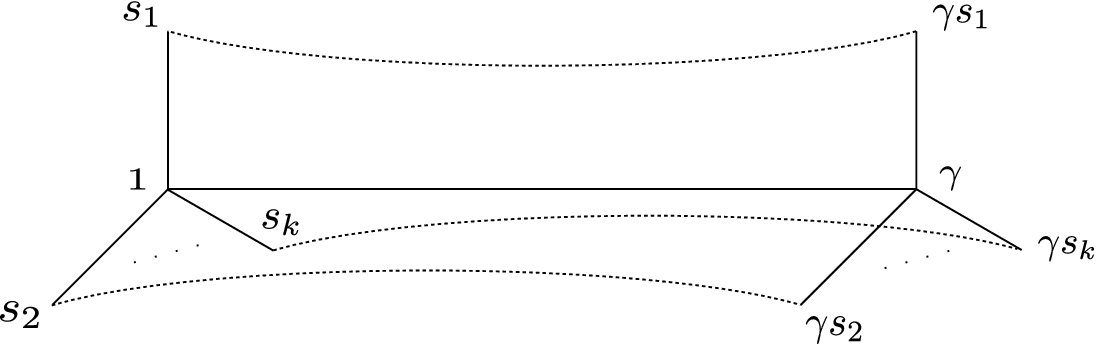}
		\label{fig:Ric-defn-group}
\end{figure}
The definition of curvature depends on the choice of the generating set $S$, but we mostly write $\kappa$ instead of $\kappa_S$ when $S$ is chosen or the context is clear. We also remark that the above definition of curvature differs by a normalization from that in \cite{NDK}, where they define the curvature to be $\bar \kappa=\kappa/|\gamma|$. For the convenience, we work with the unnormalized curvature $\kappa$. However, concerning the results of this paper, all conditions and conclusions involve only the sign of curvature, therefore the results remain valid if we replace $\kappa$ by $\overline{\kappa}$.

With the curvature for groups introduced, our first theorem now states,

\begin{theorem}\label{thm:CheegerGromoll}
Let $\Gamma$ be a group generated by a finite, symmetric set $S$. If outside a ball, the condition $\kappa\ge 0$ holds for all elements on a finite index subgroup of $\Gamma$, and there is an undistorted element $w\in \Gamma$, then $\Gamma$ is virtually abelian.
\end{theorem}

An element $w\in \Gamma$ is called undistorted if the map $\ZZ\to \Gamma$ defined by $n\mapsto w^n$ is a quasi-isometric embedding. The existence of such element is an analog condition to the existence of a bi-infinite geodesic in Riemannian manifolds. We also note that, while the non-negative curvature condition relies on the choice of the generating set, the existence of an undistorted element only depends on the group itself.

Conversely, it is shown in \cite[Theorem 12]{NDK} that for every virtually abelian group, there is a choice of a generating set such that the curvature vanishes on an abelian normal subgroup of finite index. Thus, together with a theorem of Wilking \cite[Theorem 2.1]{Wil00}, the assumption in the above theorem holds if and only if $\Gamma$ is isomorphic to the fundamental group of a closed Riemannian manifold of non-negative Ricci curvature. This gives some evidences that the two notions of Ricci curvature share similarities. We will give a more detailed discussion in Section \ref{subsec:cor-and-appl}.

Our next theorem strengthens the curvature condition, and accordingly we obtain a more refined structure on virtually abelian groups.

\begin{theorem}\label{thm:FZ}
Let $\Gamma$ be a group generated by a finite, symmetric set $S$. If $\kappa\ge 0$ holds for all elements outside a ball, and there is an undistorted element $w\in \Gamma$, then either the center of $\Gamma$ has finite index or $\Gamma$ is virtually cyclic. Conversely, if the center of $\Gamma$ has finite index or $\Gamma$ is virtually cyclic, then there exists a finite generating set $S$ such that $\kappa(x)=0$ for all $|x|>3$.
\end{theorem}

Our splitting theorems focus on the non-negative curvature conditions, and we remark that groups with strictly positive curvature has been studied in \cite[Theorem 26]{NDK}. It is shown that only finite groups can have positive curvature outside a ball. This is an analogy to Bonnet-Myers theorem, which states that a closed Riemannian manifold with strictly positive Ricci curvature has finite diameter, and thus finite fundamental group.

Now turning to the negative curvature side, we relate the curvature condition to the growth of a group. We give a positive answer to a question of \cite[Section 7]{NDK}.
\begin{theorem}\label{thm:growth}
Let $\Gamma$ be a finitely generated group, if there exists a normal subgroup $N$ and a constant $R$ such that $\kappa(x)<0$ for all $x\in N$ and $|x|>R$, then $\Gamma$ has exponential growth.
\end{theorem}

Lastly, as an application of our splitting theorems, we show an interesting curvature phenomenon for nilpotent groups. This is an analog of a well-known fact proved by Milnor \cite{Milnor76} about mixed Ricci curvatures on non-ablelian nilpotent Lie groups.

\begin{theorem}\label{thm:nilpotent} If $\Gamma$ is a torsion free nilpotent group that is not virtually abelian, then for any finite generating set, there are two sequences $(x_i)$, $(y_i)$ in $\Gamma$ both tending to infinity, which satisfy $\kappa(x_i)>0$ and $\kappa(y_i)<0$ for all $i\in \ZZ^+$.
\end{theorem}

For the 3-dimensional integral Heisenberg group with the standard generating set, a stronger statement is obtained in \cite[Theorem 16]{NDK} and \cite[Theorem 4.1]{KrophollerMallery20}, namely the sets of positive, negative and zero curvature elements all have positive densities. However, our theorem holds for a large class of nilpotent groups, and more importantly it does not depend on the finite generating set.\\

{\bf More on history.} The first attempt to define Ricci curvature on metric spaces goes back to the 70s, when Garland \cite{Garland} defined p-adic curvature on Euclidean buildings. The curvature is based on the spectral gap of Laplacian on unit spheres. Cheeger, Muller, and Schrader \cite{Cheeger84} defined Ricci curvature for piecewise flat spaces, which coincides with Ricci curvature for Riemannian manifolds and is closed under Gromov-Hausdorff convergence. Sturm \cite{SturmI,SturmII}, and later on Lott and Villani \cite{LottVillani}, defined the notion of Ricci curvature bounded from below. These definitions inherit similar infinitesimal behaviors as in Riemannian Ricci curvature. With a slightly different point of view, Ollivier \cite{Oll09} defined quantitatively the Ricci curvature for metric spaces with a random walk. Ollivier's curvature simulates the Ricci curvature in coarse scale geometry. The notion of curvature we work here was defined by Bar-Natan, Duchin, and Kropholler \cite{NDK}, also with features of coarse geometry (which they call {\it conjugation curvature}), but for groups only.\\

The paper is organized as follows. In Section \ref{sec:preliminary}, we present some background and preparation for the proofs of our theorems. We prove our main splitting theorems together with their applications in Section \ref{sect:splittingthm}. We also discuss the correspondence between finitely generated groups of non-negative curvature and fundamental groups of manifolds of non-negative Ricci curvature. In Section \ref{sec:growth}, we study the growth of groups and prove Theorem \ref{thm:growth}. We end our paper in Section \ref{sec:discuss} with some discussions and further questions.\\

\noindent {\em Acknowledgments:} We thank Jason Behrstock and Jeff Cheeger for helpful discussions. We are grateful to the Department of Mathematics at Indiana University for their hospitality while part of this work was completed. The first author also thanks the Courant Institute of Mathematical Sciences at New York University for their strong support. We thank the referee for pointing out a flaw in the proof of Theorem \ref{thm:abelian} in an earlier draft.

\section{Background}\label{sec:preliminary}

In this section we provide some preliminary background, and prove some key lemmas as a preparation for proofs of our main theorems in next sections.

\subsection{Growth of a group}
Let $\Gamma$ be a group generated by a finite symmetric generating set $S$. Then $\Gamma$ has a natural left invariant word length metric $d$. We set $|\gamma|=d(1,\gamma)$ where $1$ denotes the identity. We denote $B(n)$ the closed ball of radius $n$ centered at the identity, and similarly $S(n)$ the sphere of radius $n$ for any positive integer $n$. If $\Lambda<\Gamma$ is a subgroup, we simply denote $B_\Lambda(n)=B(n)\cap \Lambda$ and $S_\Lambda(n)=S(n)\cap \Lambda$.

We say $\Gamma$ has \emph{polynomial growth} if there exists constants $C, k>0$ such that $|B(n)|\leq C(n^k+1)$ for all integers $n\geq 1$. This property is independent on the choice of finite generating set. By a celebrated theorem of Gromov \cite{Gromov81}, a finitely generated group has polynomial growth only if it is virtually nilpotent. We say $\Gamma$ has \emph{exponential growth} if there exists a constant $a>1$ such that $|B(n)|\geq a^n$ for all integers $n\geq 1$, and again this property is independent on the choice of finite generating set. Note that for finitely generated groups we always have $|B(n)|\leq (|S|+1)^n$, so the exponential growth represents the largest possible growth of a finitely generated group. The following lemma will be used later in our proofs.

\begin{lemma}\label{lem:i-j-conjugacy}
	If $\gamma\in \Gamma$ has infinite order, and $\gamma^i$ is conjugate to $\gamma^j$ for some positive integers $i<j$, then $\Gamma$ has exponential growth.
\end{lemma}

\begin{proof}
    Let $g\in \Gamma$ be an element such that $g\gamma^i g^{-1}=\gamma^j$. Fix a constant $C_0>0$ such that $|g|\leq C_0$ and $|\gamma|\leq C_0$. We show inductively $|\gamma^n|<C\log_{j/i}(n)$ where $C$ is a constant such that $C>2C_0+jC_0$ and $|\gamma^2|<C\log_{j/i}(2)$. Suppose $|\gamma|^k<C\log_{j/i}(k)$ holds for all $k\leq n-1$. We write $n=j\cdot\floor{\frac{n}{j}}+n_0$ where $0\leq n_0<j $. By triangle inequalities, 
    
    \begin{equation*}
    	\begin{split}
|\gamma|^n\leq |(\gamma^j)^{\floor{\frac{n}{j}}}|+|\gamma^{n_0}|
&\leq |g\gamma^{i\floor{\frac{n}{j}}}g^{-1}|+n_0|\gamma|\\
&< |\gamma^{i\floor{\frac{n}{j}}}|+2|g|+j|\gamma|	\\
&\leq |\gamma^{i\floor{\frac{n}{j}}}|+2C_0+jC_0.\\
         \end{split}
    \end{equation*}
Now we use induction hypothesis,
 $$|\gamma^{i\floor{\frac{n}{j}}}|\leq C\log_{j/i}(i\floor{\frac{n}{j}})\leq C(\log_{j/i} (n)-1).$$
We recall that $C>2C_0+jC_0$. Thus, it follows that
 \begin{equation}\label{eq:distortion}
|\gamma|^n\leq C\log_{j/i} (n)-C+2C_0+iC_0<C\log_{j/i} (n).
\end{equation}

This implies that we always have distinct elements $\{1,\gamma,\gamma^2,...,\gamma^{\floor{i^{N/C}}}\}$ inside the ball $B(N)$. Hence, $\Gamma$ has exponential growth.
\end{proof}

\subsection{Distortion}

In the above proof, inequality (\ref{eq:distortion}) indicates the large distortion of $\gamma$ in the sense that the set $\{\gamma^i: i=1,\dots ,n\}$ is trapped in a much smaller ball than $B(n)$. In general, we can quantify the distortion and define the \emph{stable norm} of $\gamma$ to be 
$$\ell(\gamma)=\lim_{n\rightarrow \infty}\frac{|\gamma^n|}{n}$$
One can check that the stable norm satisfies the following properties:
\begin{enumerate}
	\item $\ell(\gamma)=\ell(\gamma^{-1})$,
	\item $\ell(\gamma)=\ell(\gamma')$ if $\gamma$ is conjugate to $\gamma'$,
	\item $\ell(\gamma)\leq |\gamma|$,
	\item $\ell(\gamma^n)=n\ell(\gamma)$,
	\item $\ell(\gamma_1\gamma_2)\leq \ell(\gamma_1)+\ell (\gamma_2)$, if $\gamma_1,\gamma_2$ commutes.
\end{enumerate}

We say an element $\gamma\in \Gamma$ is \emph{distorted} if $\ell(\gamma)=0$, and \emph{undistorted} otherwise. Being (un)distorted is independent on the choice of the finite generating set. It follows immediately from the proof of Lemma \ref{lem:i-j-conjugacy} that,

\begin{corollary}\label{cor:distortion}
	If $\gamma^i$ is conjugate to $\gamma^j$ for two distinct positive integers, then $\gamma$ is distorted.
\end{corollary}

\begin{proof}
	If $\gamma$ has finite order, then it is automatically distorted. Otherwise, use inequality (\ref{eq:distortion}).
\end{proof}

\subsection{Orbit decomposition under conjugation}

Let $\Gamma$ be any group. $\Gamma$ acts on itself by conjugation. This gives rise to an orbit decomposition $$\Gamma=\bigcup_{\gamma\in \Gamma }[\gamma],$$
where $[\gamma]$ is the $\gamma$-orbit representing the set of elements conjugate to $\gamma$. We can write $\Gamma$ as a disjoint union
$$\Gamma=\Gamma_0\cup X,$$
where $\Gamma_0=\{\gamma\in \Gamma|\;[\gamma] \;\textrm{is finite}\}$, and $X=\{\gamma\in \Gamma|\;[\gamma] \;\textrm{is infinite}\}$. Moreover, $\Gamma_0$ is a normal subgroup of $\Gamma$. On the other hand, the product of an element with infinite conjugacy class and one with finite conjugacy class has infinite conjugacy class, we will frequently use this fact in the proofs later. Hence, $\Gamma_0$ acts on $X$ by left translation.

If $N$ is a normal subgroup of $\Gamma$, then the above decomposition passes onto $N$. Since $N$ is normal, $\gamma\in N$ implies $[\gamma]\subset N$. Thus we have
$$N=N_0\cup X_{N},$$
where $N_0=\{\gamma\in N |\;[\gamma] \;\textrm{is finite}\}$ and $X_{N}=\{\gamma\in N |\;[\gamma] \;\textrm{is infinite}\}$. Note that $[\gamma]$ always means the whole $\Gamma$-orbit, which might differ from the $N$-orbit even when $\gamma\in N$.

If the subgroup is not normal, then the above decomposition fails. However, if $N$ has finite index, we can always take the normal core in order to obtain a normal subgroup.

\begin{lemma}\label{lem:subgroup-to-normal}
	If $H$ is a finite index subgroup of $G$, then there exists a finite index normal subgroup $N$ of of $G$ which is contained in $H$.
\end{lemma}

\begin{proof}
  $G$ acts by left translation on the set of cosets $G/H$. This gives rise to a representation $\rho:G\rightarrow S_n$, where $n=|G/H|$. Set $N=\ker \rho$, then $N$ is a finite index normal subgroup of $G$, and it is contained in $H$.
\end{proof}

\subsection{Non-negative curvature}

Following the above notation, we now relate the curvature conditions to some of the intrinsic properties of a group.

\begin{proposition}\label{prop:finite}
	If $\Gamma$ is finitely generated by $S$ and $N$ is a finite index normal subgroup which satisfies $\kappa(x)\geq 0$ for all $x\in N\backslash B(R)$, then 
\begin{enumerate}
  \item $X_N$ consists of only finitely many conjugacy classes, and the number of conjugacy classes is at most the number of elements in $B(R)$.
  \item $\kappa (x)= 0$ for all but finitely many $x\in N$.
  \item If $x\in X_N$, then $\ell(x)=0$.
  \item If $x\in N$ is undistorted, then the centralizer $Z_N(x)=\{\gamma\in N: \gamma x= x\gamma\}$ is contained in $N_0$.
\end{enumerate}	
	
\end{proposition}

\begin{proof} For $(1)$: It suffices to show that every element $x\in X_N$ is conjugate to an element in $B(R)$. We show inductively on the norm $|x|$. Suppose each element in $B(n-1)\cap X_N$ is conjugate to an element in $B(R)$, let $x$ be any element in $X_N$ satisfying $|x|=n$. Since $[x]$ is infinite, the conjugacy class has to leave the ball $B(n)$ at some stage, hence there exists $x'\in [x]$ and $s'\in S$ such that $|x'|\leq n$ and $|s'x's'^{-1}|>n$. On the other hand, $\kappa(x')\geq 0$ implies
	$$\sum_{s\in S}(|x'|-|sx's^{-1}|)\geq 0.$$
	Combining the two inequalities, there exists $t\in S$ such that $|tx't^{-1}|<|x'|=n$. Therefore $x$ is conjugate to $tx't^{-1}$, an element inside $B(n-1)\cap X_N$ as $N$ is normal. By induction hypothesis, it is further conjugate to an element in $B(R)$.
	
For (2): We expand the idea from \cite[Theorem 26]{NDK}. Let $r_1$, $r_2$ be two arbitrary integers satisfying $R\leq r_1<r_2-4$. Consider on $N$ the annulus $A_N(r_1,r_2)=B_{N}(r_2)\backslash B_{N}(r_1)$. We sum up all the curvatures over this annulus, and we obtain that
	\begin{equation}\label{eq:sum-of-curvature}
    \sum_{x\in A_N(r_1,r_2)}\kappa(x)=\frac{1}{|S|}\sum_{(s,x)\in S\times A_N(r_1,r_2)}(|x|-|sxs^{-1}|).
	\end{equation}
	Denote $\Delta(s,x)=|x|-|sxs^{-1}|$. By triangle inequality we know $|\Delta(s,x)|$ is an integer $\leq 2$. Moreover, if $y=s_0xs_0^{-1}$ for some $s_0\in S$, then $\Delta(s_0,x)+\Delta(s_0^{-1},y)=0$. Therefore, using the fact that $N$ is normal, the sum in the Equation \ref{eq:sum-of-curvature} is pairwise canceled except when $x$ is close to the two boundaries of the annulus, such that $(s,x)$ belongs to exactly one of the following sets
	\begin{enumerate}
		\item $Y_1=\{(s,x)\in S\times (S_N(r_1+1)\cup S_N(r_1+2)):|sxs^{-1}|\leq r_1\}$ or
		\item $Y_2=\{(s,x)\in S\times (S_N(r_2)\cup S_N(r_2-1)):|sxs^{-1}|\geq r_2+1\}$.
	\end{enumerate}
    Thus we can simplify the sum in Equation \ref{eq:sum-of-curvature} to the following
    \begin{equation}\label{eq:sum-simplified}
    \sum_{x\in A_N(r_1,r_2)}\kappa(x)=\frac{1}{|S|}\sum_{(s,x)\in Y_1\cup Y_2}\Delta(s,x).
    \end{equation}
    Notice that $Y_1$ always contributes to positive terms in Equation \ref{eq:sum-simplified} and $Y_2$ only contributes negative, also the cardinalities of the two sets are bounded by
    $$|Y_1|\leq |S|\cdot(|S_N(r_1+1)|+|S_N(r_1+2)|),$$
    $$|Y_2|\leq |S|\cdot(|S_N(r_2-1)|+|S_N(r_2)|).$$
    So we can bound Equation \ref{eq:sum-simplified} further
    $$\sum_{x\in A_N(r_1,r_2)}\kappa(x)\leq \frac{1}{|S|}\sum_{(s,x)\in Y_1}\Delta(s,x)\leq 2 (|S_N(r_1+1)|+|S_N(r_1+2)|).$$
    This shows that the sum of curvatures $\sum_{x\in A_N(r_1,r_2)}\kappa(x)$ stabilizes 
    when $r_1$ is fixed and $r_2\rightarrow \infty$. This means eventually $\kappa(x)$ has to be $0$ outside a sufficiently large ball.

For $(3)$: If $x$ has finite order, then clearly $\ell(x)=0$. If $x$ has infinite order, we look at the conjugacy classes $[x^n]$ for $n\in \ZZ$. If $x^k$ has finite conjugacy class, then $x^{k+1}=x\cdot x^k$ has infinite conjugacy class. Therefore, $[x^n]$ is infinite for infinitely many $n\in \NN$. Thus by $(1)$ and the Pigeonhole Principle, we have two distinct positive integers $i, j$ such that $x^i$ is conjugate to $x^j$. Therefore by Corollary \ref{cor:distortion}, $x$ is distorted.

For $(4)$: We prove by contradiction. Suppose $Z_N(x)\cap X_N\neq \varnothing$, and we pick any element $y$ inside the intersection. Since $x$ is undistorted, $[x]$ is finite according to $(3)$. By the choice of $y$, $[y]$ is infinite, so $[x y]$ is also infinite. This implies $\ell(y)=0$ and $\ell(xy)=0$ again by $(3)$. Since $x$ commutes with $y$, we have $\ell(x)\leq \ell(x y)+\ell(y^{-1})=0$, which gives a contradiction.
\end{proof}

\section{Splitting theorems}\label{sect:splittingthm}

\subsection{Proof of Splitting theorems}
In this section, we prove Theorem \ref{thm:CheegerGromoll} and \ref{thm:FZ}. First, we restate our Theorem \ref{thm:CheegerGromoll},

\begin{theorem}\label{thm:splitting1}
Let $\Gamma$ be a group generated by a finite, symmetric set $S$. If $\Gamma$ has an undistorted element $w$, and there is a finite index subgroup $N< \Gamma$, such that $\kappa(x)\ge 0$ for all but finitely many $x\in N$, then $\Gamma$ is virtually abelian.
\end{theorem}

\begin{proof} By Lemma \ref{lem:subgroup-to-normal}, we may assume $N$ is normal. Following the notation in Section \ref{sec:preliminary}, $N_0$ denotes the set of elements having finite conjugacy classes, and $X_N$ denote the set of elements having infinite conjugacy classes. We first show that $N_0$ is finitely generated. If $X_N=\varnothing$, then $N_0=N$, which has finite index in $\Gamma$, thus is finitely generated. Suppose $X_N\neq\varnothing$, then there exists $x\in N$ such that $[x]$ is infinite. Since $N$ has finite index in $\Gamma$, some power of $w$ is contained in $N$. We note that if $w$ is undistorted, then so are its powers. Thus, by abusing notation, we may assume $w\in N$. Since $w$ is undistorted, it has infinite order and $w\in N_0$ by $(3)$ of Proposition \ref{prop:finite}. By $(4)$ of Proposition \ref{prop:finite}, $Z_N(w)\subset N_0\subset N$. Because the conjugacy class of $w$ in $\Gamma$ is finite, the conjugacy class of $w$ in $N$ is also finite. Hence, $Z_N(w)$ has finite index in $N$. This implies $N_0$ has finite index in $N$. Thus we conclude that $N_0$ has finite index in $N$, hence also in $\Gamma$. In particular, $N_0$ is finitely generated. 

To finish, we let $t_1,...,t_k$ be a set of generators of $N_0$. For every $i=1,\dots, k$, since $[t_i]$ is finite, we have $Z_{N_0}(t_i)<N_0$ of finite index. So the center $Z(N_0)=\bigcap\limits_{i=1}^k Z_{N_0}(t_i)$ also has finite index in $N_0$. In particular, $N_0$ is virtually abelian. Therefore, $N$ is virtually abelian and so is $\Gamma$.
\end{proof}

\begin{remark}
We note that our assumption of Theorem \ref{thm:CheegerGromoll} is not vacuous, and in fact is optimal. Indeed, by \cite[Theorem 12]{NDK}, every virtually abelian group has a normal subgroup of finite index on which the curvature vanishes. Thus our theorem completely characterizes virtually abelian groups.
\end{remark}

Now turning to the proof of Theorem \ref{thm:FZ}, it is convenient to introduce the following definition.

\begin{definition}
We say $\Gamma$ is an $\op{FZ}$-group if the center $Z(\Gamma)$ has finite index in $\Gamma$.
\end{definition}

We divide Theorem \ref{thm:FZ} into two parts. The first part of Theorem \ref{thm:FZ} now states as follows.

\begin{theorem}\label{thm:abelian}
Let $\Gamma$ be a group generated by a finite, symmetric set $S$. If $\Gamma$ has an undistorted element, and there is $R>0$ such that $\kappa(x)\ge 0$ for every $x\notin B_\Gamma(R)$, then either $\Gamma$ is an $\op{FZ}$-group or $\Gamma$ is virtually cyclic.
\end{theorem}

\begin{proof}
By Theorem \ref{thm:CheegerGromoll}, we know $\Gamma$ is virtually abelian. Let $\ZZ^n$ be a finite index normal subgroup of $\Gamma$. We have the following short exact sequence
\[1\to \ZZ^n \to \Gamma\to K\to 1,\]
where $K$ is a finite group. We denote by $\Lambda$ the $\ZZ^n$ subgroup of $\Gamma$ from the short exact sequence.

We note that the natural representation $\Gamma\to \op{Aut}(\ZZ^n)=\GL_n(\ZZ)$ factors through $K\cong\Gamma/\Lambda\to \op{Aut}(\ZZ^n)$, hence the elements in the image of the adjoint representation $\rho:\Gamma\rightarrow \GL_n(\ZZ)$ all have finite order. In order to classify all possible groups, we divide them into the following three cases:

\begin{enumerate}
	\item $\rho$ is the trivial representation.
	\item $n=1$ and $\rho$ is non-trivial.
	\item $n\geq 2$ and $\rho$ is non-trivial.
\end{enumerate}

Clearly, case $(1)$ corresponds to $\op{FZ}$-groups and case $(2)$ implies that $\Gamma$ is virtually cyclic. In the rest of the proof, we show that case $(3)$ is impossible thus the theorem follows.

Suppose we are in case $(3)$. Let $\alpha\in\Gamma$ be an element such that $\rho(\alpha)$ is non-trivial. We claim that $\alpha$ does not commute with any non-trivial element in $\ZZ^n$. Since $\rho(\alpha)$ is non-trivial, there exist nontrivial elements $a, b$ of $\ZZ^n$ such that $\alpha a \alpha^{-1}=b$ where $a\neq b$. This means $\alpha$ does not commute with any powers of $a$. Hence, the set $\{a^k\alpha a^{-k}:k\in \ZZ\}$ is infinite. So then the conjugacy class $[\alpha]$ is infinite. On the other hand, $\Gamma$ is virtually $\ZZ^n$ so the inclusion $\ZZ^n\cong\Lambda\rightarrow \Gamma$ is a quasi-isometry, hence all non-trivial elements in $\Lambda$ are undistorted. For every non-trivial element $x\in \Lambda$, by $(4)$ of Proposition \ref{prop:finite} (apply when $N=\Gamma$), the centralizer $Z_\Gamma(x)$ is contained in $\Gamma_0$. Since $[\alpha]$ is infinite, $\alpha$ is not in $\Gamma_0$. Hence $\alpha\notin Z_\Gamma(x)$. In other words, $\alpha$ does not commute with any non-trivial element in $\ZZ^n$.

Let $a$ and $b$ be standard generators in $\ZZ^n$ such that they generate a subgroup $W<\ZZ^n$ that is isomorphic to $\ZZ^2$. We define a map $\Phi:W\rightarrow [\alpha]\subset\Gamma$ by $\Phi(w)=w\alpha w^{-1}$. Since $\alpha$ does not commute with any element in $\ZZ^n$, the map $\Phi$ is injective. Moreover for every $s\in \{a,a^{-1},b,b^{-1}\}$, by triangle inequality we have 
\begin{align}\Big||\Phi(w)|_\Gamma-|\Phi(ws)|_\Gamma\Big|=\Big||\Phi(w)|_\Gamma-|s\Phi(w)s^{-1}|_\Gamma\Big|\le 2|s|_\Gamma\le 2k,\label{eqn:sphere}\end{align}
where $k=\max\{|a|_\Gamma,|b|_\Gamma\}$.

Consider the conjugacy class $[\alpha]$, we form a graph $G_{\alpha}$ such that the vertices are the elements in $\Phi(W)$, we make an edge between any two vertices if they are conjugate by $a^{\pm 1}$ or $b^{\pm 1}$. Thus $\Phi$ induces a graph isomorphism between the standard Cayley graph of $\ZZ^2$ and $G_{\alpha}$. Such graph has large connectivity in the sense that the boundary of any set $A_m\supseteq B_{\ZZ^2}(m)$ satisfies $|\partial A_m|\rightarrow \infty$ as $|A_m|\rightarrow \infty$. We recall here that given a subset $A$ in a graph $G$, the boundary $\partial A$ of $A$ consists of points in $A$ that are adjacent to points in $G-A$. We show this is impossible in $\Gamma$, which leads to a contradiction. The idea is very similar to the notion of separation profile introduced in \cite{BST18}.

\begin{definition}
We define the \emph{exiting time} $\tau$ of any element $\gamma$ to be
$$\tau(\gamma)=\inf_{\sigma\in \Gamma}\{|\sigma|\;:\;|\sigma\gamma\sigma^{-1}|>|\gamma|\}$$
If there is no such element $\sigma$, we define it to be $\infty$. We say an element $\gamma$ is \emph{an exit} if $\tau(\gamma)=1$ and a \emph{$k$-step exit} if $\tau(\gamma)\leq k$.
\end{definition}
We need the following proposition.
\begin{proposition}\label{prop:finite-exit}
Fix an integer $k>0$. If $\Gamma$ satisfies $\kappa(x)=0$ for all $|x|>R$, then there is $L>0$ such that the number of $k$-step exits on each sphere is at most $L$.\end{proposition}

\begin{proof}
By definition, any $k$-step exit is conjugate to an exit by an element in the finite set $B_{\Gamma}(k)$. So we only need to show the number of exits on each sphere is uniformly bounded.

Consider the ball $B(m)$ where $m>|R|$, summing all the curvatures we obtain
$$\sum_{x\in B(m)}\kappa(x)=\frac{1}{|S|}\sum_{(s,x)\in S\times B(m)}(|x|-|sxs^{-1}|)$$
where the left hand side is independent of $m$ and the terms on the right hand side are canceled on the pairs $(s,x)$ and $(s^{-1},sxs^{-1})$ except when $x$ is near boundary, or more precisely when $(s,x)$ belongs to
$$Y=\{(s,x)\in S\times \big(S(m-1)\union S(m)\big):|sxs^{-1}|>m\} .$$
Since each pair in $Y$ either contributes $-1$, or $-2$ in the sum, we can estimate further
\begin{align*}
\sum_{x\in B(m)}\kappa(x)&=\frac{1}{|S|}\sum_{(s,x)\in Y}(|x|-|sxs^{-1}|)\leq -\frac{|Y|}{|S|}
\end{align*}
and so $|Y|\leq L$ for some uniform constant $L$.

Now if $\gamma\in S(m)$ is an exit, then by definition, there exists $s\in S$ such that $(s,\gamma)\in Y$. This shows the total number of exits on $S(m)$ is bounded by the cardinality of $Y$, hence is uniformly bounded by $L$.
\end{proof}

We continue the proof of Theorem \ref{thm:abelian}. We look at the sequence of subgraphs $G_{\alpha}^{(m)}=\Phi(W)\cap B_{\Gamma}(m)$. It is an exhaustion of $G_\alpha$. We claim $|\partial G_{\alpha}^{(m)}|$ is uniformly bounded. For any element $g\in \partial G_{\alpha}^{(m)}$, by definition, there exists $u\in \{a^{\pm 1}, b^{\pm 1}\}$ such that $|ugu^{-1}|>m$. In particular, $g$ is a $k$-step exit where we recall that $k=\max\{|a|_\Gamma,|b|_\Gamma\}$. Moreover, by inequality (\ref{eqn:sphere}), $m-2k\le |g|_\Gamma \le m$. Thus, by Proposition \ref{prop:finite-exit}, $|\partial G_{\alpha}^{(m)}|$ is uniformly bounded.

By pulling back the sequence $\{G_{\alpha}^{(m)}\}$, we obtain an exhaustion $\{A_m=\Phi^{-1}(G_{\alpha}^{(m)})\}$ on $\ZZ^2$, with $|\partial A_m|$ are uniformly bounded. It is a contradiction. Hence the case (3) is impossible. This completes proof of Theorem \ref{thm:abelian}.
\end{proof}

Finally, we prove the second part of Theorem \ref{thm:FZ}.

\begin{theorem}\label{thm:FZconverse}
Let $\Gamma$ be a finitely generated group. If $\Gamma$ is either $\op{FZ}$ or virtually cyclic, then there exists a finite generating set $S$ such that $\kappa(x)=0$ for all $x\in \Gamma\backslash B(3)$.
\end{theorem}

\begin{proof}

If $\Gamma$ is finitely generated $\op{FZ}$, we take any finite, symmetric generating set $S_0$. Consider its conjugation closure $S=\bigcup\limits_{\gamma\in S_0}[\gamma]$. For every $\gamma\in \Gamma$, the point stabilizer under the conjugation action of $\Gamma$ at $\gamma$ is exactly the centralizer $Z_\Gamma(\gamma)$ that contains the center. Thus, $|[\gamma]|=[\Gamma:Z_\Gamma(\gamma)]\leq [\Gamma:Z(\Gamma)]<\infty$. Hence the conjugacy class $[\gamma]$ is finite. It follows that $S$ is also finite. Moreover $S$  is symmetric and generates $\Gamma$. We claim that under this new generating set $S$, the curvature $\kappa_S(x)=0$ for all $x\in\Gamma$. It suffices to show $|\gamma x\gamma^{-1}|=|x|$ for all $x, \gamma\in \Gamma$. By symmetry, it is equivalent to show $|\gamma x\gamma^{-1}|\leq |x|$ for all $x, \gamma\in \Gamma$. Now given $x,\gamma\in \Gamma$ arbitrary, suppose $x$ is expressed in shortest word $x_1...x_k$ for $x_i\in S$, then $\gamma x\gamma^{-1}=(\gamma x_1\gamma^{-1})...(\gamma x_k\gamma^{-1})$ where each $\gamma x_i\gamma^{-1}\in S$ since $S$ is closed under conjugation. This implies that $|\gamma x\gamma^{-1}|\leq |x|$ thus the claim follows.

If $\Gamma$ is virtually cyclic, then by the classification theorem \cite[Lemma 11.4]{Hem04}, $\Gamma$ is either finite (which is $\op{FZ}$), finite extension of $\ZZ$, or finite extension of dihedral group $D_{\infty}=(\ZZ/2) * (\ZZ/2)$. If $\Gamma$ is a finite extension of $\ZZ$, we have the short exact sequence
$$1\rightarrow F\rightarrow \Gamma\rightarrow \ZZ\rightarrow 1,$$
where $F$ is a finite group. This sequence always splits hence $\Gamma\simeq F\rtimes \ZZ$. Let $a$ be a generator of $\ZZ$. Since $F$ is finite, there exists an integer $N>0$ such that $a^N$ acts trivially on $F$. Then the cyclic group generated by $a^N$ is a subgroup of the center which has finite index. This implies $\Gamma$ is $\op{FZ}$ hence the proposition follows.

If $\Gamma$ is a finite extension of the infinite dihedral group, then we have the following short exact sequence
$$1\rightarrow F\rightarrow \Gamma\xrightarrow{\varphi} D_{\infty}\rightarrow 1,$$
where $F$ is a finite group. Let $F=\{1,g_1,...g_k\}$ and let $a, b$ be the generators of $D_{\infty}$ viewed as $\ZZ/2\ast \ZZ/2$. We take any lifts $\bar{a}$ and $\bar b$ in $\Gamma$ such that $\varphi(\bar a)=a$ and $\varphi(\bar b)=b$. Then any element in $\Gamma$ can be written uniquely as $g_i$, or $w(\bar a,\bar b)$, or $g_iw(\bar a,\bar b)$ where $w(\bar a,\bar b)$ is an alternating word formed by $\bar a$ and $\bar b$ that belongs to the following set
$$W=\{(\bar a\bar b)^m, \bar a(\bar b\bar a)^{m-1}, (\bar b\bar a)^m, \bar b(\bar a\bar b)^{m-1}\;|\;m\in \ZZ^+ \}.$$
To see this we note that for every $i$, there exist $j$ and $m$ such that $\bar a g_i = g_j \bar a$ and $\bar b g_i = g_m\bar b$. More over for every $i$, there is $s$ and $t$ such that $\bar a^2= g_s \bar a$ and $\bar b^2 = g_t \bar b$. Thus, we can write any element if $\Gamma$ in the form $g_iw(\bar a,\bar b)$ for some $i$ and $w(\bar a, \bar b)\in W$. It  is clear that under the map $\varphi$, both $w(\bar a,\bar b)$ and $g_iw(\bar a,\bar b)$ send to the word $w(a,b)$--the same word as $w(\bar a,\bar b)$ by replacing $\bar a$, $\bar b$ by $a$, $b$ respectively.

Now we set $S=\varphi^{-1}(\{1,a,b\})\backslash \{1\}$, or more explicitly,
$$S=\{g_1,...,g_k,\bar a, g_1\bar a,...,g_k\bar a, \bar b, g_1\bar b,...,g_k\bar b\}$$
with no repetition. It is clear that $S$ is a finite, symmetric generating set of $\Gamma$. We claim $\varphi$ is norm-preserving when restricted to $\Gamma\backslash S$. Pick any $\gamma\in \Gamma\backslash S$, write $\gamma$ in shortest word $s_1...s_l$ where $s_i\in S$, in particular, $|\gamma|=l$. Then $\varphi(\gamma)=\varphi(s_1)...\varphi(s_l)$ where by construction each $\varphi(s_i)$ is either $1$, $a$ or $b$. This shows that $|\gamma|\geq |\varphi (\gamma)|$ for all $\gamma \in \Gamma$. On the other hand, since $\varphi(\gamma)$ is not trivial, it can be expressed in shortest alternating word $w(a,b)$, then there exists $g_i\in F$ such that $\gamma=g_iw(\bar a,\bar b)$. Now if we spell $w(\bar a,\bar b)$ out, the first letter (either $\bar a$ or $\bar b$) can be combined with the $g_i$ in the front which still lies in $S$, and this way it produces a spelling of length same as $|w(a,b)|$, so this implies that $|\gamma|\leq |\varphi(\gamma)|$ provided $\gamma\in \Gamma\backslash S$. Thus, we conclude $|\gamma|=|\varphi(\gamma)|$ for all $|\gamma|>1$.

Finally we compare the curvatures in $(\Gamma,S)$ with that in $(D_\infty,\{a,b\})$ under the homomorphism $\varphi$. It is clear $\kappa(y)=0$ for all $|y|>1,y\in D_\infty$. Now for any $x\in \Gamma$ with $|x|>3$, the curvature at $x$ by definition satisfies
\begin{align*}
|S|\cdot\kappa(x)&=\sum_{i=1}^k (|x|-|g_ix g_i^{-1}|)+\sum_{i=1}^k (|x|-|(g_i\bar a)x (g_i\bar a)^{-1}|)\\
&+\sum_{i=1}^k (|x|-|(g_i\bar b)x (g_i\bar b)^{-1}|)+(|x|-|\bar a x \bar a^{-1}|)+(|x|-|\bar b x \bar b^{-1}|).
\end{align*}
Notice by triangle inequalities, all the norms appearing in the above sum is at least $1$, so by the fact $\varphi$ is norm preserving, we can replace all norms by the norms of the corresponding images under $\varphi$. So we obtain the following
\begin{align*}
|S|\cdot\kappa(x)&=\sum_{i=1}^k (|\varphi(x)|-|\varphi(x)|)+\sum_{i=1}^k (|\varphi(x)|-|a\varphi(x) a^{-1}|)\\
&+\sum_{i=1}^k (|\varphi(x)|-|b\varphi(x) b^{-1}|)+(|\varphi(x)|-|a\varphi(x) a^{-1}|)\\
&+(|\varphi(x)|-|b\varphi(x) b^{-1}|)\\
&\\
&=2(k+1)\kappa(\varphi(x))=0.\\
\end{align*}
Thus $\kappa(x)=0$ for all $x\in \Gamma\backslash B(3)$, and this completes the proof.
\end{proof}

\subsection{Corollaries and applications}\label{subsec:cor-and-appl}
We have seen that both Theorem \ref{thm:CheegerGromoll} and Theorem \ref{thm:FZ} require the existence of an undistorted element in the group. It is known that infinite hyperbolic groups always have at least one such element, so our Theorem \ref{thm:FZ} immediately implies the following corollary.
\begin{corollary}\label{cor:hyp}
If $\Gamma$ is a non-elementary $\delta$-hyperbolic group, then for any choice of generating set, the set of elements with negative curvature is infinite.\qed
\end{corollary}

A similar result to the above corollary also holds for nilpotent groups. In fact, with a bit more careful study on the nilpotent structure, we obtain a much stronger statement as follows,

\begin{theorem} \label{thm:nil}If $H$ is a torsion-free nilpotent group which is not virtually abelian, then for any finite generating set, there are two sequences of elements $(x_i)$, $(y_i)$ in $H$ both tending to infinity, which satisfy $\kappa(x_i)>0$ and $\kappa(y_i)<0$ for all $i\in \NN$.
\end{theorem}

\begin{proof}
An infinite finitely generated nilpotent group always contains an undistorted element. By Theorem \ref{thm:abelian}, there is a sequence $(y_i)$ of negative curvature. To show the existence of $(x_i)$, we suppose, for the sake of contradiction, that there are only finitely many elements in $H$ having positive curvature, i.e.~there exists $R>0$ such that $\kappa(x)\leq 0$ for all $x\in H\backslash B(R)$.

Let $H=H_m\triangleright...\triangleright H_2\triangleright H_1=\{1\}$ be the lower central series of $H$ which satisfies $[H,H_i]=H_{i-1}$ for all integers $2\leq i\leq m$. We note that $m\geq 3$ because $H$ is not abelian. We also have $H_2\subset Z(H)$ and $H_3\backslash Z(H)\neq \varnothing$. We fix an arbitrary element $w\in H_3\backslash Z(H)$ and consider the conjugacy class $[w]$. Since $w\in H_3\backslash Z(H)$, there are $h\in H$ and $1\neq z\in Z(H)$ such that $hwh^{-1}w^{-1}=z$. It follows that $h^nwh^{-n}=z^nw$, for all $n\in \NN$. Since $H$ is torsion free, we have $\langle z\rangle$ is a free cyclic subgroup. Hence, $[w]$ is infinite.

Let $S=\{s_1,...,s_k\}$ be  a symmetric generating set. We recall that an element $x$ is called an exit if there exists $s\in S$ such that $|sxs^{-1}|>|x|$. We pick an exit $u_0\in[w]$ with $|u_0|>R$ and denote $z_i=[s_i,u_0]$ for $1\leq i\leq k$. Let $g\in H$ such that $u_0=gwg^{-1}$. Then
$$z_i=[s_i,u_0]=s_igwg^{-1}s_i^{-1}gw^{-1}g^{-1}=[s_ig,w][w,g],$$
for every $i=1,\dots, k$. Since $[H,w] \subset H_2 \subset Z(H)$, we have $z_i\in Z(H)$ for every $i=1,\dots , k$. This implies the following bracket relation
$$z_i^{n^2}=[s_i^n,u_0^n],$$
for every $i=1,\dots, k$. It follows that all $z_i$ are distorted. 

Next, we construct a sequence $(u_n)$ of exits and a sequence $(s_{i_n})$ such that for all $n\in \NN$
\begin{enumerate}
\item $s_{i_n}\in S$, 
\item $u_{n+1}=s_{i_{n+1}}u_{n}s_{i_{n+1}}^{-1}$,
\item $|u_{n+1}|>|u_n|$.
\end{enumerate}
Suppose, by induction, that we can construct $u_0,\dots, u_n$ and $s_{i_1},\dots, s_{i_n}$. Since $u_n$ is an exit, there is an element $s_{i_{n+1}}\in S$ such that $|s_{i_{n+1}}u_{n}s_{i_{n+1}}^{-1}|>|u_n|$. We set $u_{n+1}=s_{i_{n+1}}u_{n}s_{i_{n+1}}^{-1}$. Since $\kappa(u_{n+1})\le 0$ and $|s_{i_{n+1}}^{-1}u_{n+1}s_{i_{n+1}}|=|u_n|<|u_{n+1}|$, $u_{n+1}$ must be an exit. This justifies the construction. As a consequence of property (3), we have $|u_n|\geq |u_0|+n$. In particular, $\liminf\limits_{n\rightarrow \infty}\frac{|u_n|}{n}\geq 1$.

On the other hand, we can write $u_n=s_{i_n}...s_{i_1}u_0s_{i_1}^{-1}...s_{i_n}^{-1}=z_{i_1}...z_{i_n}u_0$, where $z_{i_j}$ belongs to the set $\{z_1,...,z_k\}$ for every $j\in \NN$. Since $z_i\in Z(H)$ for every $i=1,\dots, k$, we can write further in repeated powers that $u_n=z_1^{\ell_1}...z_{k}^{\ell_k}u_0$ with $\sum\limits_{i=1}^k \ell_i=n$. So we have
\begin{align}\label{ineq:nil1}
|u_n|&\leq |z_1^{\ell_1}|+...+|z_k^{\ell_k}|+|u_0|.
\end{align}
Since $z_i$ are all distorted, we have that for every $\epsilon>0$, there exists $N(\epsilon)$ such that if $\ell>N$, then $|z_i^\ell|<\epsilon|z_i|\ell$ for all $i$. If $\ell\leq N$, we have the trivial bound $|z_i^\ell|\leq N|z_i|$. Thus, combining the two, we have for all $i$ and $\ell,$ that
\begin{align}\label{ineq:nil2}
|z_i^\ell|\leq \epsilon|z_i|\ell+N|z_i|.
\end{align}
Let $L=\max\{|z_1|,...,|z_k|\}$, we choose $\epsilon=\frac{1}{2L}$ and let $N=N(\frac{1}{2L})$. By inequalities  \ref{ineq:nil1} and \ref{ineq:nil2}, we obtain
\begin{align*}
|u_n|-|u_0|&\leq (\frac{\ell_1}{2}+NL)+...+(\frac{\ell_k}{2}+NL)\\
&=\frac{n}{2}+kNL.
\end{align*}
it follows that $\limsup\limits_{n\rightarrow \infty}\frac{|u_n|}{n}\leq \frac{1}{2}$, which contradicts with the fact $\liminf\limits_{n\rightarrow \infty}\frac{|u_n|}{n}\geq 1$ we obtained before. Therefore, there are infinitely many elements in $H$ of positive curvature.
\end{proof}

We remark that it is shown in \cite{NDK} that, for the Heisenberg group, when equipped with its standard generators, the set of points with positive curvature and the set of points with negative curvature both have positive density. So the question is whether this mixed curvature behavior will occur in a more general context. Theorem \ref{thm:nil} shows that the answer is yes for the class of torsion-free nilpotent groups, and is independent of the generating set. This is an analog to the fact that non-abelian nilpotent Lie groups with left invariant metrics have mixed signs of Ricci curvature \cite[Theorem 2.4]{Milnor76}.

As another application of our splitting theorems, we can compare groups that arise from the two different notions of curvature, that is,
finitely generated groups of non-negative curvature and the fundamental groups of closed manifolds with non-negative Ricci curvature. In particular, we get the following correspondence.

\begin{proposition}Let $\Gamma$ be a finitely generated group. The followings are equivalent.
\begin{enumerate}
\item \label{1}$\Gamma$ is isomorphic to the fundamental group of a closed Riemannian manifold of non-negative sectional curvature.
\item \label{2}$\Gamma$ is isomorphic to the fundamental group of a closed Riemannian manifold of non-negative Ricci curvature.
\item \label{3}$\Gamma$ is isomorphic to the fundamental group of a closed Riemannian manifold of positive Ricci curvature.
\item \label{4}$\Gamma$ has an undistorted element, and there is a generating set $S$ that makes the curvature $\kappa_S$ restricted to a finite index subgroup vanish everywhere outside a ball.
\item \label{5}There are a finite group $F$, a crystallographic group $\Lambda$ and a short exact sequence
$$1\to F\to \Gamma\to \Lambda \to 1.$$
\item There are a finite group $K$, an integer $n$, and a short exact sequence
$$1\to \ZZ^n\to \Gamma\to K\to 1.$$
\end{enumerate}
\end{proposition}

\begin{proof}The equivalence of (\ref{1}), (\ref{2}), (\ref{3}), (5), and (6) is obtained from \cite[Theorem 2.1, Theorem 2.3]{Wil00}. By Theorem \ref{thm:splitting1}, we have $(4)\implies (6)$. By \cite[Theorem 12]{NDK}, we have $(6)\implies (4)$.
\end{proof}

We also have the following correspondence for compact homogeneous spaces.

\begin{proposition}Let $\Gamma$ be a finitely generated group that is not virtually cyclic. The followings are equivalent.
\begin{enumerate}
\item $\Gamma$ is isomorphic to the fundamental group of a compact homogeneous space.
\item $\Gamma$ is isomorphic to the fundamental group of a homogeneous space.
\item $\Gamma$ has an undistorted element, and there is a generating set $S$ that makes the curvature $\kappa_S$ vanish everywhere outside a ball.
\item There is a finite group $F$, an integer $n$ and a short exact sequence
$$1\to F\to \Gamma\to \ZZ^n \to 1.$$
\item The center of $\Gamma$ has finite index.
\end{enumerate}
\end{proposition}

\begin{proof}The equivalence of ({1}), ({2}), (4), and ({5}) is obtained from \cite[Theorem 2.2]{Wil00}. Since $\Gamma$ is finitely generated and is not virtually cyclic, by Theorem \ref{thm:abelian} and Theorem \ref{thm:FZconverse}, we have (3) is equivalent to (5).
\end{proof}

\section{Growth}\label{sec:growth}

In this section we turn to the questions concerning the growth of groups. For Riemannian manifolds, the Bishop-Gromov inequality indicates close relations between the volume growth of balls and the Ricci curvature. In particular, if an $n$-dimensional manifold has $\op{Ric}\geq -k(n-1)$, then the volume of the balls of radius $r$ is bounded above by that in the same dimensional model space of constant sectional curvature $-k$. 

In the last section of \cite{NDK}, Bar-Natan, Duchin, and Kropholler asked if having negative curvature outside a ball would imply the exponential growth of the group. We answer this affirmatively by the following theorem.

\begin{theorem}
	Let $\Gamma$ be an infinite finitely generated group. If there exists a normal subgroup $N$ satisfying $\kappa(x)<0$ for all  $x\in N\backslash B(R)$ for some $R>0$, then $\Gamma$ has exponential growth.
\end{theorem}

\begin{proof}
	We use a similar technique as in the proof of $(2)$ of Proposition \ref{prop:finite}. Let $r_1$, $r_2$ be two arbitrary integers satisfying $R\leq r_1<r_2-4$, Equation \ref{eq:sum-simplified} still holds:
	$$\sum_{x\in A_N(r_1,r_2)}\kappa(x)=\frac{1}{|S|}\sum_{(s,x)\in Y_1\cup Y_2}\Delta(s,x).$$
	Since $Y_1$ contributes positively in the sum, we can bound the sum from below,
	$$\sum_{x\in A_N(r_1,r_2)}\kappa(x)\geq \frac{1}{|S|}\sum_{(s,x)\in Y_2}\Delta(s,x)\geq -2(|S_N(r_2)|+|S_N(r_2+1)|).$$
	We set $r_1=R$. By the definition of the curvature, $\kappa(x)<0$ implies $\kappa(x)\leq -1/|S|$. Hence,
	\begin{equation}\label{eq:sphere-bound-ball}
|S_N(r_2)|+|S_N(r_2+1)|\geq \frac{1}{2}\sum_{x\in A_N(R,r_2)}\frac{1}{|S|}=\frac{1}{2|S|}|A_N(R,r_2)|,
	\end{equation}
	for all $r_2>R+4$. 
	Use $|B_N(n)|=|B_N(n-2)|+|S_N(n-1)|+|S_N(n)|$, we obtain from inequality \ref{eq:sphere-bound-ball} that
	\begin{equation*}
		\begin{split}
	|B_N(n)|&\geq |B_N(n-2)|+\frac{1}{2|S|}|A_N(R,n-1)|\\ &=|B_N(n-2)|+\frac{1}{2|S|}|B_N(n-1)|-\frac{1}{2|S|}|B_N(R)|\\
	&\geq (1+\frac{1}{2|S|})|B_N(n-2)|-\frac{1}{2|S|}|B_N(R)|.\\
	\end{split}
	\end{equation*}
So we have
\begin{align*}
|B_N(n)|-|B_N(R)|\geq (1+\frac{1}{2|S|})\Big(|B_N(n-2)|-|B_N(R)|\Big).
\end{align*}	
Iterate the above inequality, we obtain
    $$|B_N(n)|\geq C_1\Bigg(\sqrt{1+\frac{1}{2|S|}}\Bigg)^n+C_2$$
    for some $C_1, C_2>0$ which does not depend on $n$. This implies $\Gamma$ has exponential growth.
\end{proof}

\begin{remark}
	In the Riemannian setting, the condition $\op{Ric}<0$ is quite flexible. In fact, every smooth manifold with dimension $\geq 3$ admits a negative Ricci metric \cite{Lohk94}.
\end{remark}

Using Lemma \ref{lem:subgroup-to-normal}, we get the following corollary.

\begin{corollary}
    Given a finitely generated group $\Gamma$, if there exists a finite index subgroup $\Gamma'$ satisfying $\kappa(x)<0$ for all but finitely many $x\in \Gamma'$, then $\Gamma$ has exponential growth.
\end{corollary}

\section{Further discussions}\label{sec:discuss}
We have seen that a group must have exponential growth if it admits a generating set so that the curvature is negative everywhere except finitely many elements. On the other hand, if a group admits a generating set that makes curvature positive everywhere except finitely many elements, then it is finite \cite[Theorem 26]{NDK}. So it is natural to ask the same question in the context of non-negative curvature.
\begin{question} What is the upper bound of the growth of a group with non-negative curvature everywhere except for finitely many elements?
\end{question}

This question can be treated as an analog for Bishop-Gromov inequality. If the answer for this question is polynomial growth, by Gromov's polynomial growth theorem \cite{Gromov81}, then such a group is virtually nilpotent. As a corollary, we can remove the condition about the existence of an undistorted element in both Theorem \ref{thm:splitting1} and Theorem \ref{thm:abelian}. We remark that the existence of a bi-infinite geodesic is essential in Cheeger-Gromoll's splitting theorem, but for the universal cover of a closed Riemannian manifold, it is automatic. Thus, we can ask,

\begin{question} Is it possible to remove the assumption on the existence of an undistorted element in both Theorem \ref{thm:splitting1} and Theorem \ref{thm:abelian}?
\end{question}

Another natural question to ask is the relation between this notion of negative curvature and $\delta$-hyperbolicity or $CAT(0)$ property for groups. This has already been asked in \cite{NDK}. Examples of $\delta$-hyperbolic groups which have infinitely many elements of positive curvature, were constructed in \cite{NDK}. By Corollary \ref{cor:hyp}, we also know that they have infinitely many elements of negative curvature. However, all these examples have torsion. We do not know if there is an example that is torsion free. In particular, we can ask,
\begin{question}\label{que:free}Is it true that for every generating set, the free group $F_2$ has negative curvature everywhere except finitely many elements?
\end{question}
We note that a similar argument in \cite[Example 15]{Oll09} shows that, for every non-elementary $\delta$-hyperbolic group, there is a choice of a generating set such that the group has only finitely many elements of non-negative Ollivier's curvature, and thus of the curvature we consider.

\begin{remark}
By the time this paper is revised, Question \ref{que:free} has been answered by Andrew Keisling \cite[Theorem 1]{Keisling21}. Namely, there exists a generating set of $F_2$ such that the set of zero curvature elements has positive density.
\end{remark}

\end{document}